\documentclass[10pt]{article}
\usepackage{color}
\usepackage{amssymb}

\usepackage{amsthm,array,amssymb,amscd,amsfonts,latexsym, url}
\usepackage{amsmath}

\newtheorem{theo}{Theorem}[section]
\newtheorem{prop}[theo]{Proposition}
\newtheorem{lemm}[theo]{Lemma}
\newtheorem{coro}[theo]{Corollary}
\newtheorem{rema}[theo]{Remark}
\newtheorem{Defi}[theo]{Definition}
\newtheorem{question}[theo]{Question}

\newtheorem{conj}[theo]{Conjecture}

\title{The generalized Hodge  and Bloch conjectures  are equivalent for  general complete intersections, II}
\author{Claire Voisin
\\CNRS, \'{E}cole Polytechnique}
\date{}

\newfont{\gothic}{eufb10}
\begin{document}
\maketitle

\begin{abstract} We prove an unconditional (but slightly weakened) version of the main result of \cite{voisingenhodgebloch},
which was, starting from dimension $4$, conditional to the Lefschetz
standard conjecture. Let $X$ be a variety with trivial Chow groups,
(i.e. the cycle class map to cohomology is injective on
$CH(X)_\mathbb{Q}$). We prove that if the cohomology of a general
hypersurface $Y$ in $X$ is ``parameterized by cycles of dimension
$c$'', then the Chow groups $CH_{i}(Y)_\mathbb{Q}$ are trivial for
$i\leq c-1$.

 \end{abstract}

Let $X$ be a smooth projective variety. We will say that $X$ has
geometric coniveau $\geq c$ if the transcendental cohomology of $X$,
that is, the orthogonal with respect to Poincar\'e duality of the
``algebraic cohomology'' of $X$  generated by classes of algebraic
cycles,
$$H^*(X,\mathbb{Q})_{tr}:=H^*(X,\mathbb{Q})_{alg}^{\perp},$$
is supported on a closed algebraic subset $W\subset X$,
with ${\rm codim}\,W\geq c$.

According to the generalized Hodge conjecture
\cite{grothhodge}, $X$ has geometric coniveau $\geq c$ if and only if $X$ has Hodge coniveau
$\geq c$, where we define the Hodge coniveau of $X$ as the minimum over $k$ of the
Hodge coniveaux of the Hodge structures $H^k(X,\mathbb{Q})_{tr}$. Here we recall
that the Hodge coniveau of a weight $k$ Hodge structure $(L,L^{p,q})$ is the
integer $c\leq k/2$ such that
$$L_\mathbb{C}=L^{k-c,c}\oplus L^{k-c-1,c+1}\oplus\ldots \oplus L^{c,k-c}$$
with $L^{k-c,c}\not=0$. As the Hodge coniveau is computable by looking at  the Hodge numbers,
we know conjecturally how to compute the geometric coniveau.

A fundamental conjecture on algebraic cycles is the generalized
Bloch conjecture (see \cite[Conjecture 1.10]{voisinweyl}), which was
formulated by Bloch \cite{bloch} in the case of surfaces, and can be
stated as follows:
\begin{conj} \label{conjbloch} Assume $X$ has geometric coniveau $\geq  c$. Then the cycle class map
$$CH_i(X)_\mathbb{Q}\rightarrow H^{2n-2i}(X,\mathbb{Q}),\,n={\rm dim}\,X,$$
is injective for any $i\leq c-1$.
\end{conj}

Concrete examples are given by hypersurfaces in projective space, or
more generally complete intersections. For a smooth complete
intersection $Y$ of $r$ hypersurfaces in $\mathbb{P}^n$, the Hodge
coniveau of $Y$ is equal to the Hodge coniveau of
$H^{n-r}(Y,\mathbb{Q})_{tr}$, the last space being for the very
general member $Y$, except in a small number of cases, equal to the
Hodge coniveau of $H^{n-r}(Y,\mathbb{Q})_{prim}$. The latter is
computed by Griffiths:
\begin{theo} If $Y\subset \mathbb{P}^n$ is a complete intersection of hypersurfaces
of degrees $d_1\leq\ldots\leq d_r$, the Hodge coniveau of $H^{n-r}(Y,\mathbb{Q})_{prim}$
is $\geq c$ if and only if $n\geq \sum_id_i+(c-1)d_r$.
\end{theo}
Conjecture \ref{conjbloch} thus predicts that for such a $Y$, the
Chow groups $CH_i(Y)_\mathbb{Q}$ are equal to $\mathbb{Q}$ for
$i\leq c-1$, a result which is essentially known only for coniveau
$1$ (then $Y$ is a Fano variety, so $CH_0(Y)=\mathbb{Z}$) and a
small number of particular cases for coniveau $\geq 2$, eg cubic
hypersurfaces of dimension $\leq 8$ or complete intersections of
quadrics \cite{otwi}.

We will say that a  smooth projective variety $X$ has trivial Chow
groups if for any $i$, the cycle class map
$CH_i(X)_\mathbb{Q}\rightarrow H^{2n-2i}(X,\mathbb{Q}),\,n={\rm
dim}\,X,$ is injective. By \cite{lewis}, this implies that the whole
rational  cohomology of $X$ is algebraic, that is, consists of cycle
classes. The class of such varieties includes projective spaces and
more generally toric varieties, Grassmannians, projective bundles
over a variety with trivial Chow groups, see
\cite{voisingenhodgebloch} for further discussion of this notion.

In \cite{voisingenhodgebloch}, we proved  Conjecture \ref{conjbloch}
for very general  complete intersections of very ample hypersurfaces
in a smooth projective variety $X$ with trivial Chow groups, {\it
assuming the Lefschetz standard conjecture}.  More precisely, the
results proved in loc. cit. are unconditional in the case of
surfaces and threefolds, for which the Lefschetz standard conjecture
is not needed. They have been  improved later on for families of
surfaces in \cite{voisincampedelli}, where the geometric setting is
much more general: instead of the universal family of complete
intersection surfaces, we consider any  family of smooth projective
 surfaces $\mathcal{S}\rightarrow B$  satisfying the condition that
$\mathcal{S}\times_B\mathcal{S}\rightarrow B$ has a smooth
projective completion with is rationally connected or more generally
has trivial $CH_0$ group.

The purpose of this paper is to prove unconditionally, in the
geometric setting
 of general complete intersections $Y$ in a variety with trivial Chow groups $X$,
 a slightly weaker form of
Conjecture \ref{conjbloch}, which is equivalent to it in dimension
$2,\,3$, or assuming the Lefschetz standard conjecture.

Assume $Y$ has dimension $m$ and geometric coniveau $c$. Then there
exist a smooth projective  variety $W$ with ${\rm dim}\,W=m-c$, and
a morphism $j:W\rightarrow Y$ such that
$j_*:H^{m-2c}(W,\mathbb{Q})\rightarrow H^m(Y,\mathbb{Q})_{tr}$ is
surjective. This follows from the definition of the geometric
coniveau and from Deligne's results on mixed Hodge structures
\cite{deligneII} (see \cite[proof of Theorem 2.39]{voisinweyl}). Let
us now introduce a stronger notion, which is in fact equivalent to
having geometric coniveau
 $\geq c$  if we the Lefschetz standard
conjecture (see \cite[Section 1]{voisingenhodgebloch}).
\begin{Defi} \label{defiparam} Let $Y$ be
smooth projective of dimension $m$. We will say that the degree $m$
cohomology of $Y$ (or its primitive part
 with respect to a polarization) is parameterized by algebraic cycles of dimension $c$ if

 a)  There
 exist a smooth projective variety $T$ of dimension $m-2c$ and a correspondence
$P\in CH^{m-c}(T\times Y)_\mathbb{Q}$, such that
$$P^*: H^m(Y,\mathbb{Q})\rightarrow H^{m-2c}(T,\mathbb{Q})$$
is injective (or equivalently:
$P_*:H^{m-2c}(T,\mathbb{Q})\rightarrow H^m(Y,\mathbb{Q}) $ is
surjective), resp.
$$\,\,P^*: H^m(Y,\mathbb{Q})_{prim}\rightarrow H^{m-2c}(T,\mathbb{Q})$$
is injective.

b)   Furthermore $P^*$ is compatible up to a coefficient with the
intersection forms: for some rational number $N\not=0$,
$<P^*\alpha,P^*\beta>_T=N<\alpha,\beta>_Y$ for any
$\alpha,\,\beta\in H^m(Y,\mathbb{Q})$, (resp. for any
$\alpha,\,\beta\in H^m(Y,\mathbb{Q})_{prim}$).

\end{Defi}

\begin{rema}{\rm The condition a) in Definition \ref{defiparam}
obviously implies that $H^m(Y,\mathbb{Q})$ has geometric coniveau $\geq c$, since it
vanishes away from the image in $Y$  of the support of $P$, which is
of dimension $\leq m-c$. The more precise  condition that
$H^m(Y,\mathbb{Q})$ comes from the cohomology of a variety $T$ of
dimension $\leq m-2c$ is formulated explicitly in \cite{vial}, where
it is shown that the two conditions are equivalent assuming the
Lefschetz standard conjecture. Our definition is still stronger
since we also impose the condition b) concerning the comparison of
the intersection forms.

}
\end{rema}
\begin{rema} \label{rema30jan} {\rm Assume the Hodge structure on $H^m(Y,\mathbb{Q})_{prim}$ is simple and exactly
of Hodge coniveau $c$. Assume furthermore it does not admit other
polarizations than  the multiples of the one given by $<\,,\,>_Y$.
Then the nontriviality
 of $P^*: H^m(Y,\mathbb{Q})_{prim}\rightarrow H^{m-2c}(T,\mathbb{Q})$ implies its injectivity
 by the simplicity of the Hodge structure and also
 the  condition b) of compatibility with the cup-product. Indeed,
by assumption, $H^{m-c,c}(Y)_{prim}\not=0$ hence by injectivity of
$P^*$, we get nonzero classes $P^*\alpha\in H^{m-2c,0}(T)$. By the
second Hodge-Riemann bilinear relations \cite[2.2.1]{voisinweyl}, we
then have $<P^*\alpha,P^*\overline{\alpha}>_T\not=0$. Thus the
pairing $<P^*\alpha,P^*\beta>_T$ on the Hodge structure
$H^{m}(Y,\mathbb{Q})_{prim}$ is nondegenerate and polarizes this
Hodge structure. Hence it must be by uniqueness a nonzero rational
multiple of the pairing $<\,,\,>_Y$ and thus, condition b) is
automatically satisfied in this case.

}

\end{rema}
Actually, we will use in the paper a reformulation of
 Definition \ref{defiparam} (see Lemma \ref{lediagZY}). Namely,
assuming that the cohomology of $Y$ splits as the orthogonal direct
sum
$$H^*(Y,\mathbb{Q} )=K\bigoplus_{\perp}
H^m(Y,\mathbb{Q})_{prim},\,\,K\subset H^*(Y,\mathbb{Q})_{alg},$$ our
set of conditions a) and b) for primitive cohomology is equivalent
to the fact that there is a cohomological decomposition of the
diagonal
$$[\Delta_Y]=[Z]+\sum_i\alpha_{i}[Z_i\times Z'_i]\,\,{\rm in}\,\,H^{2m}(Y\times Y,\mathbb{Q}),$$
where $Z_i,\,Z'_i$ are algebraic subvarieties of $Y$, ${\rm
dim}\,Z_i+{\rm dim}\,Z'_i=m$, and $Z$ is an $m$-cycle of $Y\times Y$
which is supported on $W\times W$, where $W\subset Y$ is a closed
algebraic subset with ${\rm dim}\,W\leq m-c$.

 The main result we prove in
this paper is:
\begin{theo} \label{theomain}Let $X$ be a smooth projective $n$-fold with trivial Chow groups and
let $L$ be a very ample line bundle on $X$. Assume that for the
general hypersurface $Y\in |L|$, the cohomology  group
$H^{n-1}(Y,\mathbb{Q})_{prim} $ is nonzero and parameterized by
algebraic cycles of dimension $c$ in the sense of Definition
\ref{defiparam}.

 Then for any
smooth member $Y$ of $|L|$, the cycle class map
$$CH_i(Y)_\mathbb{Q}\rightarrow H^{2n-2-2i}(Y,\mathbb{Q}),\,n={\rm dim}\,X,$$
is injective for any $i\leq c-1$.
\end{theo}
\begin{rema}{\rm One can more generally consider a very ample vector bundle $E$ on $X$ and
the smooth varieties $Y\subset X$ of codimension $r={\rm rank}\,E$
obtained as zero loci of sections of $E$. This however immediately
reduces to the hypersurface case by replacing $X$ with
$\mathbb{P}(E^*)$, (see \cite[4.1.2]{voisinweyl} for details). }
\end{rema}
\begin{rema}{\rm The condition that $H^{n-1}(Y,\mathbb{Q})_{prim} $ is nonzero is not very restrictive:
very ample hypersurfaces with no nonzero primitive cohomology are
rather rare (even if they exist, for example odd dimensional
quadrics in projective space). Typically, if $X$ is defective, that
is, its projective dual is not a hypersurface, its hyperplane
sections have no nonzero primitive cohomology. We refer to
\cite{zak1}, \cite{zak2} for the study of this phenomenon. }
\end{rema}

We will give in section \ref{secap} one concrete application of
Theorem \ref{theomain}. It concerns hypersurfaces obtained as
hyperplane sections of the Grassmannian $G(3,10)$ which were studied
in \cite{devoi}.

We will finally conclude the paper explaining how to modify the
assumptions of Theorem \ref{theomain} in order to cover cases where
the line bundle $L$ is not very ample (see Proposition
\ref{propmild}, Theorem \ref{theovariant}). This is necessary if we
want to apply these methods to submotives of $G$-invariant
hypersurfaces cut-out by a projector of $G$, where $G$ is a finite
group acting on $X$.

Let us say a word on the strategy of the proof. First of all, our
assumption can be reformulated by saying that an adequate correction
$\Delta_{Y,prim}$ of the diagonal $\Delta_Y$ of $Y$ by a cycle
restricted from $X\times X$ is cohomologous to a cycle $Z$ supported
on $W\times W$, where $W\subset Y$ is a closed algebraic subset of
codimension $\geq c$.

We then deduce from the fact that this last property is satisfied by
a general $Y\in |L|$ that an adequate correction $\Delta_{Y,prim}$
of the diagonal $\Delta_Y$ of $Y$ by a cycle restricted from
$X\times X$ is {\it rationally equivalent}  to a cycle $Z$ supported
on $W\times W$, where $W\subset Y$ is a closed algebraic subset of
codimension $\geq c$. We finally use the following lemma (see
\cite{voisingenhodgebloch}):
\begin{lemm} \label{introle} Assume $X$ has trivial Chow groups and
that we have a decomposition
\begin{eqnarray}\label{eqintrole}
\Delta_Y=Z_1+Z_2\,\,{\rm in}\,\,CH^{n-1}(Y\times Y)_\mathbb{Q},
\end{eqnarray}
where $Z_1$ is the restriction of a cycle on $X\times X$ and
$Z_2$ is supported on $W\times W$, with ${\rm codim}\,W\geq c$, then
$CH_i(Y)_{\mathbb{Q},hom}=0$ for $i\leq c-1$.
\end{lemm}
\begin{proof} For any $z\in CH_i(Y)_{\mathbb{Q},hom}$, let both sides
of (\ref{eqintrole}) act on $z$.
We then get
$$z=Z_{1*}z+Z_{2*}z\,\,{\rm in}\,\,CH_{i}(Y)_\mathbb{Q}.$$
As $Z_1$ is the restriction of a cycle on $X\times X$, the map
$Z_{1*}$ on  $CH_{i}(Y)_{\mathbb{Q},hom}$ factors through
$j_*:CH_i(Y)_{\mathbb{Q},hom}\rightarrow CH_i(X)_{\mathbb{Q},hom}$
and $ CH_i(X)_{\mathbb{Q},hom}$ is $0$ by assumption. On the other
hand, if $i\leq c-1$, $Z_{2*}z=0$ because the projection of the
support of $Z_2$ to $Y$ is of codimension $\geq c$ so  does not meet
a general representative of $z$.
\end{proof}

\section{Proof of Theorem \ref{theomain}\label{secproof}}
We establish a few preparatory lemmas before giving the proof of the
main theorem. Let $X$ be a smooth projective variety of dimension
$n$ with trivial Chow groups, and $L$ be a very ample line bundle on
$X$. Let $Y\subset X$ be a smooth  member of $|L|$.

We start with the following lemma:
\begin{lemm} \label{lediagZY}  Let
$T$ be a smooth projective variety of dimension $n-1-2c$ and $P\in
CH^{n-1-c}(T\times Y)_{\mathbb{Q}}$  such that
$$P^*:H^{n-1}(Y,\mathbb{Q})_{prim}\rightarrow H^{n-2c-1}(T,\mathbb{Q})$$
is compatible with cup-product up to a coefficient, that is
\begin{eqnarray}
\label{eqcuppro}(P^*\alpha,P^*\beta)_T=N(\alpha,\beta)_Y,\,\,\forall
\alpha,\,\beta\in H^{n-1}(Y,\mathbb{Q})_{prim}, \end{eqnarray} for
some $N\not=0$. Then
\begin{eqnarray}
\label{eqform30jan1}(P,P)_*([\Delta_{T}])=N[\Delta_{Y}]+[\Gamma]+[\Gamma_1]\,\,{\rm
in} \,\,H^{2n-2}(Y\times Y,\mathbb{Q}),\end{eqnarray} where the
cycle $\Gamma$ is the restriction to $Y\times Y$ of a cycle with
$\mathbb{Q}$-coefficients on $X\times X$, and the cycle $\Gamma_1$
is $0$ if $n-1$ is odd, and of the form $\sum_{i}\alpha_iZ_i\times
Z'_i$, ${\rm dim }\,Z_i={\rm dim }\,Z'_i=\frac{n-1}{2}$ if $n-1$ is
even.
\end{lemm}
Here $(P,P)\in CH^{2n-2}(T\times T\times Y\times Y)_\mathbb{Q}$ is
just the product $P\times P\subset T\times Y\times T\times Y\cong
T\times T\times Y\times Y$ if $P$ is the class of a subvariety, and
is defined as $pr_{13}^*P\cdot pr_{24}^*P$ in general.

\begin{proof} Indeed, let $\Gamma':=(P,P)_*(\Delta_{T})\in CH_{n-1}(Y\times
Y)_\mathbb{Q}$. Observe that $\Gamma'= {^tP}\circ P$ in
$CH_{n-1}(Y\times Y)_\mathbb{Q}$. As
$P^*:H^{n-1}(Y,\mathbb{Q})_{prim}\rightarrow
H^{n-2c-1}(Z,\mathbb{Q})$ satisfies (\ref{eqcuppro}), we find that
the cycle class $[\Gamma']\in H^{2n-2}(Y\times Y,\mathbb{Q}) $
 satisfies the property that
$$[\Gamma']_*=P_*\circ P^*:H^{*}(Y,\mathbb{Q})\rightarrow H^{*}(Y,\mathbb{Q})$$
induces $$NI_d: H^{n-1}(Y,\mathbb{Q})_{prim}\rightarrow
H^{n-1}(Y,\mathbb{Q})_{prim}$$ via the composite map
$$End\,(H^{n-1}(Y,\mathbb{Q}))\stackrel{rest}{\rightarrow}
Hom\,(H^{n-1}(Y,\mathbb{Q})_{prim},H^{n-1}(Y,\mathbb{Q}))$$
$$\stackrel{proj}{\rightarrow}
Hom\,(H^{n-1}(Y,\mathbb{Q})_{prim},H^{n-1}(Y,\mathbb{Q})_{prim}),$$
where the projection $H^{n-1}(Y,\mathbb{Q})\rightarrow
H^{n-1}(Y,\mathbb{Q})_{prim}$ is the transpose with respect to the
intersection pairing of the inclusion
$H^{n-1}(Y,\mathbb{Q})_{prim}\rightarrow H^{n-1}(Y,\mathbb{Q})$. As
$H^{n-1}(Y,\mathbb{Q})=H^{n-1}(Y,\mathbb{Q})_{prim}\oplus_{perp}H^{n-1}(X,\mathbb{Q})_{\mid
Y}$, it follows that
$${[\Gamma']_*}_{\vert H^{n-1}(Y,\mathbb{Q})_{prim}}=N Id+\eta:H^{n-1}(Y,\mathbb{Q})_{prim}\rightarrow
H^{n-1}(Y,\mathbb{Q}),$$ where $\eta$ takes value in
$H^{n-1}(X,\mathbb{Q})_{\mid Y}$.

 To conclude, we use the orthogonal
decomposition given by the Lefschetz theorem on hyperplane sections
$$H^{*}(Y,\mathbb{Q})\cong H^*(X,\mathbb{Q})_{\mid Y}\bigoplus_{\perp} H^{n-1}(Y,\mathbb{Q})_{prim}.$$
 The class of the  symmetric cycle
\begin{eqnarray}
\label{eqform30jan2}\Gamma'':=\Gamma'-N\Delta_{Y}=(P,P)_*(\Delta_{T})-N\Delta_{Y}
\end{eqnarray}
acts as $0$ on $H^{n-1}(Y,\mathbb{Q})_{prim}$, hence by the
orthogonal decomposition above, it lies in
$$H^*(X,\mathbb{Q})_{\mid Y}\otimes
H^*(X,\mathbb{Q})_{\mid Y}\bigoplus
H^{n-1}(Y,\mathbb{Q})_{prim}\otimes H^{n-1}(X,\mathbb{Q})_{\mid Y}$$
$$
\bigoplus H^{n-1}(X,\mathbb{Q})_{\mid Y}\otimes
H^{n-1}(Y,\mathbb{Q})_{prim}.$$

Finally we use the fact that $X$ has trivial Chow groups, so that
its cohomology is algebraic by \cite{lewis}; hence
$H^*(X,\mathbb{Q})_{\mid Y}\otimes H^*(X,\mathbb{Q})_{\mid Y}$
consists of classes of cycles on $Y\times Y$ restricted from
$X\times X$. In the decomposition above, we thus find that
\begin{eqnarray}
\label{eqform30jan3}[\Gamma'']=[\Gamma]+\eta+ \eta', \end{eqnarray}
for some classes  $\eta\in H^{n-1}(Y,\mathbb{Q})_{prim}\otimes
H^{n-1}(X,\mathbb{Q})_{\mid Y},\,\eta'\in
H^{n-1}(X,\mathbb{Q})_{\mid Y}\otimes H^{n-1}(Y,\mathbb{Q})_{prim}$,
and $[\Gamma]\in H^*(X,\mathbb{Q})_{\mid Y}\otimes
H^*(X,\mathbb{Q})_{\mid Y}$ for some algebraic cycle $\Gamma$ on
$Y\times Y$ restricted from $X\times X$. Note that if $n-1$ is odd,
then $H^{n-1}(X,\mathbb{Q})_{\mid Y}=0$, so $\eta=\eta'=0$ and we
get
$$[\Gamma'']=[(P,P)_*(\Delta_{T})]-N[\Delta_Y]=[\Gamma]$$
so the lemma is proved in this case.

 When $n-1$ is even, for $\gamma\in
H^{n+1}(X,\mathbb{Q})\cong H^{n-1}(X,\mathbb{Q})^*$, we have
$$(\eta+\eta')_*(\gamma)=\eta_*(\gamma),\,\,(\eta+\eta')^*(\gamma)={\eta'}^*(\gamma).$$
As $\eta+\eta'$ is an algebraic class on $Y\times Y$  and $\gamma$
is also algebraic, we conclude that $\eta^*(\gamma)$ is algebraic on
$Y$ for any $\gamma\in H^{n+1}(X,\mathbb{Q})$ and similarly for
${\eta'}_*(\gamma)$. It follows immediately that both classes $\eta$
and $\eta'$ can be written as $\sum_{i}\alpha_iZ_i\times Z'_i$,
${\rm dim }\,Z_i={\rm dim }\,Z'_i=\frac{n-1}{2}$, which provides by
(\ref{eqform30jan2}) and (\ref{eqform30jan3}) the desired cycle
$\Gamma_1$ with class $\eta+\eta'$, satisfying (\ref{eqform30jan1}).

\end{proof}

\begin{coro}\label{corouseful29jan} Under the same assumptions,
there is a closed algebraic subset $W\subset Y$ of codimension $\geq
c$,  an $n-1$-cycle $Z\subset W\times W$, with
$\mathbb{Q}$-coefficients and an $n-1$-cycle $\Gamma$ in $Y\times Y$
which is the restriction of an $n+1$-cycle in $X\times X$ such that
\begin{eqnarray}\label{decompdiag} [\Delta_Y]=[Z]+[\Gamma]\,\,{\rm
in}\,\,H^{2n-2}(Y\times Y,\mathbb{Q}).
\end{eqnarray}

\end{coro}
\begin{proof} Indeed, if $n-1$ is odd, we have
the equality
$$(P,P)_*([\Delta_{T}])=N[\Delta_{Y}]+[\Gamma]$$
and $(P,P)_*([\Delta_{T}])$ is supported on $W\times W$, where $W$
is the image of the support of $P$, hence has dimension $\leq
n-1-c$.

When $n-1$ is even, we write as in (\ref{eqform30jan1})
$$(P,P)_*([\Delta_{T}])=N[\Delta_{Y}]+[\Gamma]+[\Gamma_1],$$
where $\Gamma_1=\sum_i\alpha_iZ_i\times Z'_i$, with ${\rm
dim}\,Z_i=\frac{n-1}{2}$, and we take for $W$ the union of the image
of the support of $P$ and of the $Z_i$ and $Z'_i$. (This works
because $c\leq \frac{n-1}{2}$.)
\end{proof}
Let now  $B\subset |L|$ be the Zariski open set parameterizing
smooth hypersurfaces $Y_b$ in $X$ with equation
$\sigma_b\in\mathbb{P}( H^0(X,L))$ and let
$\pi:\mathcal{Y}\rightarrow B$ be the universal family,
$$\mathcal{Y}=\{(t,x)\in B\times X,\,x\in \mathcal{Y}_t\},\,\pi=pr_1.$$
 We will be mainly interested in the fibered self-product
 $\mathcal{Y}\times_B\mathcal{Y}$ where the relative diagonal $\Delta_\mathcal{Y}$
 lies, but it is more convenient to blow it up in
 $\mathcal{Y}\times_B\mathcal{Y}$. The resulting
 variety
 $\widetilde{\mathcal{Y}\times_B\mathcal{Y}}$ was also considered in
 \cite{voisingenhodgebloch} and  the following lemma was proved (we include the proof for completeness):
 \begin{lemm}\label{leselfprod} The quasi-projective variety $\widetilde{\mathcal{Y}\times_B\mathcal{Y}}$ is a Zariski open set
 in a projective bundle $M$ over the blow-up $\widetilde{X\times X}$ of $X\times X$
 along its diagonal.
 \end{lemm}
\begin{proof}Indeed, a point in $\widetilde{\mathcal{Y}\times_B\mathcal{Y}}$ is a $4$-uple $(b,x_1,x_2,z)$
consisting of  a point of $B$,  two points $x_1,\,x_2$ in
$\mathcal{Y}_b$, and  a  length $2$ subscheme   $z\subset
\mathcal{Y}_b$ whose associated  cycle is  $x_1+x_2$. There is thus
a morphism $p$ from $\widetilde{\mathcal{Y}\times_B\mathcal{Y}}$ to
$\widetilde{X\times X}$ which parameterizes  triples $(x_1,x_2,z)$
where  $x_1,\,x_2$ are two points  of $X$, and $z\subset X$ is a
subscheme of length $2$  whose associated cycle is $x_1+x_2$. The
  fiber of $p$ over $(x_1,x_2,z)$ is clearly the set
of $b\in B$ such that $\sigma_{b\mid z}=0$. Thus
$\widetilde{\mathcal{Y}\times_B\mathcal{Y}}$ is Zariski open in the
variety
$$M:=\{(\sigma,(x_1,x_2,z)),\,\sigma_{\mid z}=0\}\subset \mathbb{P}( H^0(X,L))\times \widetilde{X\times X}.$$
The very ampleness of $L$ guarantees that $M$ is a projective bundle over
$\widetilde{X\times X}$.
\end{proof}
We now  assume that the main assumption of Theorem \ref{theomain}
holds, namely that there exist for general $b\in B$ a variety $T_b$
of dimension $n-1-2c$ and a correspondence with
$\mathbb{Q}$-coefficients $P_b\in CH^{n-1-c}( T_b\times
\mathcal{Y}_b)_\mathbb{Q}$ of codimension $n-1-c$ (a family of
$c$-cycles in $Y_b$ parameterized by $T_b$) such that
$$P_b^*:H^{n-1}(Y_b,\mathbb{Q})_{prim}\rightarrow H^{n-2c-1}(T_b,\mathbb{Q})$$
is compatible with cup-product up to a coefficient $N\not=0$. We
then have the following result   in the same spirit as Proposition
2.7 in \cite{voisingenhodgebloch}, which is very simple but
nevertheless a key point in the whole argument.
\begin{lemm}\label{newlemma} Under the same assumptions, there exist a
quasi-projective variety $\mathcal{T}\rightarrow B$ and a
codimension $n-1-c$ cycle $\mathcal{P}\in {
CH}^{n-1-c}(\mathcal{T}\times_B\mathcal{Y})_\mathbb{Q}$ such that
for general $b\in B$, the map
$\mathcal{P}^*_b:H^{n-1}(\mathcal{Y}_b,\mathbb{Q})_{prim}\rightarrow
H^{n-2c-1}(\mathcal{T}_b,\mathbb{Q})$ is compatible with cup-product
up to a coefficient $N'\not=0$.
\end{lemm}
\begin{proof}[Proof] The reason is very simple: Using our assumption and  a Hilbert schemes or Chow varieties argument,
we can certainly construct data $\mathcal{T}',\,\mathcal{P}'$  as
above over a finite cover $U'$, say of degree $N_0$, of a Zariski
open set $U$ of $B$. We then consider $\mathcal{T}'$ as a family
over $U$ which we denote by $\mathcal{T}_U$, and $\mathcal{P}'$ as a
relative correspondence over $U$ between  $\mathcal{T}_U$ and
$\mathcal{Y}_U$ which we denote by  $\mathcal{P}\in CH^{n-1-c}(
\mathcal{T}_U\times_U\mathcal{Y}_U)_\mathbb{Q}$. For a general point
$u\in U$, the fiber of $\mathcal{T}_U$ over $u$ is the disjoint
union of the fibers $\mathcal{T}'_{u'}$, where $u'\in U'$ maps to
$u$, and the correspondence $\mathcal{P}_u$ is the disjoint union of
the correspondences $\mathcal{P}'_{u'}\in
CH^{n-1-c}(\mathcal{T}'_{u'}\times \mathcal{Y}_u)$, where $u'\in U'$
maps to $u$. Hence
$\mathcal{P}_u^*:H^{n-1}(\mathcal{Y}_u,\mathbb{Q})_{prim}\rightarrow
H^{n-2c-1}(\mathcal{T}_u,\mathbb{Q})$ multiplies the intersection
form by $NN_0$, which proves the lemma with $N'=NN_0$.
\end{proof}
\begin{coro}\label{coronouveau25jan}
 Under the same assumptions, there is a closed algebraic  subset $\mathcal{W}\subset \mathcal{Y}$ of codimension $\geq
c$ and a cycle $\mathcal{Z}\in CH^{n-1}(\mathcal{Y}\times_B
\mathcal{Y})_\mathbb{Q}$ which is supported on $\mathcal{W}\times_B
\mathcal{W}$, such that for any $b\in B$, the restricted cycle
$$\mathcal{Z}_b-\Delta_{\mathcal{Y}_b}$$
is cohomologous in $\mathcal{Y}_b\times \mathcal{Y}_b$ to a cycle
$\Gamma_b$ coming from $X\times X$.
\end{coro}
\begin{proof} With notation as in Lemma \ref{newlemma}, we first define $\mathcal{W}_0\subset\mathcal{Y}$ as the
image of the support of $\mathcal{P}$ under the second projection.
Then we define $\mathcal{Z}_0$ as
$\frac{1}{N'}(\mathcal{P},\mathcal{P})_*(\Delta_{\mathcal{T}/B})$,
where   $(\mathcal{P},\mathcal{P})\in
CH^{2n-2}(\mathcal{T}\times_B\mathcal{T}\times_B\mathcal{Y}\times_B\mathcal{Y})_\mathbb{Q}$
denotes the relative correspondence $pr_{13}^*\mathcal{P}\cdot
pr_{24}^*\mathcal{P}$ between $\mathcal{T}\times_B\mathcal{T}$ and
$\mathcal{Y}\times_B\mathcal{Y}$, with
$$pr_{13},\, pr_{24}: \mathcal{T}\times_B\mathcal{T}\times_B\mathcal{Y}\times_B\mathcal{Y}
\rightarrow \mathcal{T}\times_B\mathcal{Y}$$ the two natural
projections. If $n-1$ is odd, the conclusion then follows directly
from Lemma \ref{lediagZY}, with $\mathcal{Z}=\mathcal{Z}_0$,
$\mathcal{W}=\mathcal{W}_0$.

When $n-1$ is even, we argue as in the proof of Corollary
\ref{corouseful29jan}, which says that for any $b\in B$, there exist
cycles $Z_{i,b},\,Z'_{i,b},\,i\geq1,$ of dimension $\frac{n-1}{2}$
in $\mathcal{Y}_b$, a cycle $\Gamma_b$ in $\mathcal{Y}_b\times
\mathcal{Y}_b$ which is the restriction of a cycle in $X\times X$,
and rational numbers $\alpha_i$ such that
$\mathcal{Z}_{0,b}-\Delta_{\mathcal{Y}_b}-\Gamma_b$ is cohomologous
in $\mathcal{Y}_b\times \mathcal{Y}_b$ to
$\sum_i\alpha_iZ_{i,b}\times Z'_{i,b}$. The cycles
$Z_{i,b},\,Z'_{i,b},\,i\geq1,$ can be defined over a generically
finite cover $B'\rightarrow B$, giving families
$$\mathcal{Z}_i\subset \mathcal{Y}',\,\,\mathcal{Z}'_i\subset
\mathcal{Y}'$$ with $\mathcal{Y}':=\mathcal{Y}\times_BB'$. Then,
over $B'$, we have the cycle   $\mathcal{Z}'_0\in
CH^{n-1}(\mathcal{Y}'\times_{B'}\mathcal{Y}')_\mathbb{Q}$ defined as
the pull-back of $\mathcal{Z}_0$, such that for any $b\in B'$,
$$[\mathcal{Z}'_{0,b}]-[\Delta_{\mathcal{Y}'_b}-\Gamma_b]=\sum_{i\geq1}\alpha_i[\mathcal{Z}_{i,b}\times_{B'} \mathcal{Z}'_{i,b}].$$
 Denote $\phi:\mathcal{Y}'\rightarrow \mathcal{Y}$, $(\phi,\phi):\mathcal{Y}'\times_{B'}\mathcal{Y}'
 \rightarrow \mathcal{Y}\times_{B}\mathcal{Y}$ the natural
 morphisms,
$$\mathcal{W}:=\mathcal{W}_0\cup\phi(\cup_i
{\rm Supp}\,\mathcal{Z}_i)\cup\phi(\cup_i{\rm
Supp}\,\mathcal{Z}'_i),$$ and
$$\mathcal{Z}=\mathcal{Z}_0-\frac{1}{{\rm
deg}\,\phi}\,\,(\phi,\phi)_*(\sum_i\alpha_i\mathcal{Z}_{i}\times_{B'}\mathcal{Z}'_{i}).$$
Then $\mathcal{W}$ and $\mathcal{Z}$ satisfy the desired conclusion.

\end{proof}
\begin{proof}[Proof of Theorem \ref{theomain}]
Recall the Zariski open inclusion
$$\widetilde{\mathcal{Y}\times_B\mathcal{Y}}\subset M$$
of Lemma \ref{leselfprod}, where $p:M\rightarrow \widetilde{X\times
X}$ is a projective bundle over $\widetilde{X\times X}$. In both
cases, the ``$\widetilde{\,\,\,\,\,}$'' means that we blow-up along
the diagonal.

By Corollary \ref{coronouveau25jan},
our assumptions  give a
subvariety $\mathcal{W}\subset \mathcal{Y}$ of codimension $\geq c$
and a cycle $\mathcal{Z}\in CH^{n-1}(\mathcal{Y}\times_B
\mathcal{Y})_\mathbb{Q}$ which is supported on $\mathcal{W}\times_B
\mathcal{W}$, such that for any $b\in B$, the cycle
$$\mathcal{Z}_b-\Delta_{\mathcal{Y}_b}$$
is cohomologous in $\mathcal{Y}_b\times \mathcal{Y}_b$ to a cycle
$\Gamma_b$ coming from $X\times X$. Note that we can clearly assume
that $\Gamma_b$ is the restriction to $\mathcal{Y}_b\times
\mathcal{Y}_b$ of a cycle $\Gamma'$ of $X\times X$, which is
independent of $b$, since we are interested only in its cohomology
class:
$$[\Gamma_b]=[\Gamma'_{\mid \mathcal{Y}_b\times \mathcal{Y}_b}]\,\,{\rm in}\,\,H^{2n-2}(\mathcal{Y}_b\times \mathcal{Y}_b,\mathbb{Q}).$$

In other words, the cycle
$$\mathcal{Z}-\Delta_{\mathcal{Y}/B}-p_0^*(\Gamma')
\in CH^{n-1}(\mathcal{Y}\times_B\mathcal{Y})_\mathbb{Q},$$
where $p_0:\mathcal{Y}\times_B\mathcal{Y}\rightarrow X\times X$ is the natural map,
is cohomologous to $0$ along the fibers
of $\mathcal{Y}\times_B\mathcal{Y}\rightarrow B$.

We  now blow-up the relative diagonal, pull-back these cycles to
$\widetilde{\mathcal{Y}\times_B\mathcal{Y}}$  and extend them to $M$.
This provides us with a cycle
\begin{eqnarray}\label{T}
R:=\overline{\widetilde{\mathcal{Z}}}-\overline{\widetilde{\Delta_{\mathcal{Y}/B}}}-
p^*(\Gamma') \in CH^{n-1}(M)_\mathbb{Q},
\end{eqnarray}
which has the property that its restriction to
$\widetilde{\mathcal{Y}_b\times \mathcal{Y}_b}\subset M$ is
cohomologous to $0$, for any $b\in B$. We prove now:
\begin{prop} \label{lecrux} There exists a cycle
$\gamma\in CH^{n-1}(X\times X)_\mathbb{Q}$ such that for any $b\in
B$, $R-p^*\gamma$ maps to $0$ in $CH^{n-1}(\mathcal{Y}_b\times
\mathcal{Y}_b)_\mathbb{Q}$ via the map $\tau_{b*}\circ i_b^*$, where
$\tau_b:\widetilde{\mathcal{Y}_b\times \mathcal{Y}_b}\rightarrow
\mathcal{Y}_b\times \mathcal{Y}_b$ is the blow-up of the diagonal
and $i_b: \widetilde{\mathcal{Y}_b\times \mathcal{Y}_b} \rightarrow
M$ is the inclusion map.

\end{prop}
Admitting the proposition temporarily, the proof of Theorem
\ref{theomain} concludes as follows: For any $b\in B$, the image
$\tau_{b*}\circ i_b^*(R-p^*\gamma)\in CH^{n-1}(\mathcal{Y}_b\times
\mathcal{Y}_b)_\mathbb{Q}$ is by construction the cycle
\begin{eqnarray}
\label{formcycle} \mathcal{Z}_b-\Delta_{\mathcal{Y}_b}-\Gamma'_{\mid
\mathcal{Y}_b\times \mathcal{Y}_b}-\gamma_{\mid \mathcal{Y}_b\times
\mathcal{Y}_b}\in CH^{n-1}(\mathcal{Y}_b\times
\mathcal{Y}_b)_\mathbb{Q}.
\end{eqnarray}
Proposition  \ref{lecrux} says that the cycle (\ref{formcycle})
vanishes in $CH^{n-1}(\mathcal{Y}_b\times
\mathcal{Y}_b)_\mathbb{Q}$, which can be rewritten as:
\begin{eqnarray}\label{autreformcycle} \Delta_{\mathcal{Y}_b}=
\mathcal{Z}_b+\gamma'_{\mid \mathcal{Y}_b\times \mathcal{Y}_b}\,\,{\rm in}\,\,
CH^{n-1}(\mathcal{Y}_b\times \mathcal{Y}_b)_\mathbb{Q},
\end{eqnarray}
for a cycle $\gamma'\in CH^{n-1}(X\times X)_\mathbb{Q}$. Recalling
that for general $b\in B$, $\mathcal{Z}_b$ is supported on
$\mathcal{W}_b\times \mathcal{W}_b$ with $\mathcal{W}_b\subset
\mathcal{Y}_b$ closed algebraic of codimension $\geq c$, this
implies by Lemma \ref{introle} that the cycle class map is injective
on $CH_i(\mathcal{Y}_b)_\mathbb{Q}$ for general $b\in B$ and $i\leq
c-1$.

To conclude that this holds also for any $b\in B$, we can  observe
that (\ref{autreformcycle}) holds for any $b\in B$ and it is still
true for any $b\in B$ that $\mathcal{Z}_b$ is rationally equivalent
to a cycle supported on $\mathcal{W}'_b\times \mathcal{W}'_b$ with
$\mathcal{W}'_b\subset \mathcal{Y}_b$ closed algebraic of
codimension $\geq c$, even if $\mathcal{W}_b$ itself is not of
codimension $\geq c$.

\end{proof}
\begin{proof}[Proof of Proposition \ref{lecrux}] Let $\delta\in CH^1(M)$ be the class of the pull-back to
$M$ of the  exceptional divisor of $\widetilde{X\times X}$ and let
$h=c_1(\mathcal{O}_M(1))\in CH^1(M)$, where $\mathcal{O}_M(1)$
refers to the projective bundle structure of $M$ over
$\widetilde{X\times X}$. Note that $M\subset |L|\times
\widetilde{X\times X}$, where the first projection restricts on
$\mathcal{Y}\times_B\mathcal{Y}$ to the natural map to $B$. Thus $h$
is the inverse image of a line bundle on $|L|$ by the first
projection $M\rightarrow |L|$ and it restricts to $0$ in
$CH^1(\widetilde{\mathcal{Y}_b\times \mathcal{Y}_b})$. The class
$\delta$ restricts to the class $\delta_b$ of the exceptional
divisor of $\widetilde{\mathcal{Y}_b\times \mathcal{Y}_b}$. Finally,
note that
$$\tau_{b*}(\delta_b^{k})=0\,\,{\rm in}\,\,CH(\mathcal{Y}_b\times
\mathcal{Y}_b)_\mathbb{Q}\,\,{\rm for} \,\,0<k<n-1,$$
$$\tau_{b*}(\delta_b^{n-1})=(-1)^{n-2}\Delta_{\mathcal{Y}_b}\,\,{\rm
in}\,\,CH^{n-1}(\mathcal{Y}_b\times \mathcal{Y}_b)_\mathbb{Q}.$$

The projective bundle formula tells us that $CH(M)$ is generated by
the powers of $h$ as a module over the ring $CH(\widetilde{X\times
X})$. Next, as the diagonal restriction map $CH(X\times
X)\rightarrow CH(X)$ is surjective, the blow-up formula tells us
that $CH(\widetilde{X\times X})$ is generated over the ring
$CH(X\times X)$ by the powers of $\delta$.

It follows that codimension $n-1$ cycles on $M$ can be written in
the form $$z=\sum_{r,s}h^r\delta^s p^*(\gamma_{r,s}),$$ where $r+s
\leq n-1$ and $\gamma_{r,s} \in CH^{n-1-r-s}(X\times X)$. By the
above arguments, we get
$$\tau_{b*}\circ
i_b^*(z)=\gamma_{0,n-1}(-1)^{n-2}\Delta_{\mathcal{Y}_b}+{\gamma_{0,0}}_{\mid
\mathcal{Y}_b\times \mathcal{Y}_b}\,\,{\rm
in}\,\,CH^{n-1}(\mathcal{Y}_b\times \mathcal{Y}_b)_\mathbb{Q},$$
where $\gamma_{0,n-1}\in CH^{0}(X\times X)=\mathbb{Z}$ is just a
number.

We apply this analysis to the cycle $R$ of (\ref{T}), whose image in
$CH^{n-1}(\mathcal{Y}_b\times \mathcal{Y}_b)_\mathbb{Q}$ is by
construction cohomologous to $0$. Writing as above
$R=\sum_{r,s}h^r\delta^s p^*(\gamma_{r,s})$,
 this  gives  us an equality
\begin{eqnarray}\label{eqcyclescohzero}
\tau_{b*}\circ
i_b^*(R)=\gamma_{0,n-1}(-1)^{n-2}\Delta_{\mathcal{Y}_b}+{\gamma_{0,0}}_{\mid
\mathcal{Y}_b\times \mathcal{Y}_b}\,\,{\rm
in}\,\,CH^{n-1}(\mathcal{Y}_b\times \mathcal{Y}_b)_\mathbb{Q}
\end{eqnarray}
  and in particular an equality of cycle classes:
\begin{eqnarray}\label{cohzero}
 \gamma_{0,n-1}(-1)^{n-2}[\Delta_{\mathcal{Y}_b}]+[{\gamma_{0,0}}_{\mid \mathcal{Y}_b\times \mathcal{Y}_b}]=0\,\,{\rm in}
\,\,H^{2n-2}(\mathcal{Y}_b\times \mathcal{Y}_b,\mathbb{Q}).
\end{eqnarray}
Using our hypothesis that the primitive cohomology of
$\mathcal{Y}_b$ is nonzero, (\ref{cohzero}) implies that
$\gamma_{0,n-1}=0$. Thus
the image of $R$ in
$CH^{n-1}(\mathcal{Y}_b\times \mathcal{Y}_b)_\mathbb{Q}$ is equal to
${\gamma_{0,0}}_{\mid \mathcal{Y}_b\times \mathcal{Y}_b}$. This
proves the lemma, with $\gamma=\gamma_{0,0}$.

\end{proof}
\section{An application \label{secap}}
Let us  give one new application: In \cite{devoi}, Debarre and the
author studied smooth members $Y$ of $|\mathcal{L}|$, where
$\mathcal{L}$ is the Pl\"ucker polarization on the Grassmannian
$G(3,10)$. More precisely, let $V_{10}$ be a $10$-dimensional
complex vector space. To a smooth hypersurface $Y\subset
G(3,V_{10})$ defined by an element $\sigma$ of
$\bigwedge^3V_{10}^*=H^0(G(3,V_{10}),\mathcal{L})$, we associated
the subvariety  $F(Y) $ of the Grassmannian $G(6,V_{10})$ of
$6$-dimensional vector subspaces of $V_{10}$,
 defined
by
$$F(Y):=\{[W]\in G(6,V_{10}), \,W\subset V_{10},\,\sigma_{\mid W}=0\}.$$
We proved in \cite{devoi} that for general $\sigma$, $F(Y)$ is a
smooth hyper-K\"{a}hler $4$-fold. There is a natural correspondence
$P\subset F(Y)\times Y$ defined by
$$P=\{([W],[W'])\in F(Y)\times Y,\,W'\subset W\}.$$
By the first projection
$P\rightarrow F(Y)$, $P$ is a bundle over $F(Y)$ into Grassmannians
$G(3,6)$.

The following result  is proved in \cite{devoi}:
\begin{theo} The map $P^*: H^{20}(Y,\mathbb{Q})_{prim}\rightarrow H^2(F(Y),\mathbb{Q})$
is injective with image equal to $H^2(F(Y),\mathbb{Q})_{prim}$,
(where ``prim'' refers now to the Pl\"ucker polarization).
Furthermore, $h^{11,9}(Y)\not=0$,  the number of moduli of $F(Y)$ is
$20$, and this is equal to $h^{1,1}(F(Y))-1$.
\end{theo}
We have the following consequence: \begin{lemm}\label{lemm26jan2014}
There exists a number $\mu\not=0$ such that for general $Y$ (so that
$F(Y)$ is smooth of dimension $4$), we have
\begin{eqnarray}
\label{formint}<\alpha,\beta>_Y=\mu
<P^*\alpha,l^2P^*\beta>_{F(Y)},\,\,{\rm for} \,\,\,\alpha,\,\beta\in
H^{20}(Y,\mathbb{Q})_{prim}
\end{eqnarray}
 where $l=c_1(\mathcal{L}_{\mid F(Y)})\in H^2(F(Y),\mathbb{Q})$.
\end{lemm}
\begin{proof} Since the morphism
$P^*:H^{20}(Y,\mathbb{Q})_{prim}\rightarrow H^2(F(Y),\mathbb{Q})$ is
locally constant when $Y$ deforms in the family, it suffices to
prove the statement for a single very general $Y$.
 Since $F(Y)$ is a projective  hyper-K\"ahler fourfold with $20$
moduli and  $h^{1,1}(F(Y))=21$, for very general $Y$, the Hodge
structure on $H^2(F(Y),\mathbb{Q})_{prim}$ is simple, and admits an
unique polarization up to a coefficient. Hence the same is true for
the Hodge structure on $H^{20}(Y,\mathbb{Q})_{prim}$. Thus the
polarizations on both sides of (\ref{formint}) must coincide via
$P^*$ up to a nonzero coefficient.
 \end{proof}
\begin{coro}\label{corointro} The varieties $Y$ as above have their cohomology parameterized by
cycles of dimension $9$.
\end{coro}
\begin{proof} Indeed, let $T\subset F(Y)$ be the intersection of two general
members of $\mathcal{L}_{\mid F(Y)}$. Then (\ref{formint}) says that
the restricted correspondence $P_T:=P_{\mid T\times Y}$ satisfies
$$<\alpha,\beta>_Y=\mu <P_T^*\alpha,P_T^*\beta>_{T}.$$
\end{proof}

 We now get the following conclusion:
\begin{theo} The smooth hyperplane sections $Y$ of $G(3,10)$ satisfy
$CH_i(Y)_{\mathbb{Q},hom}=0$ for $i<9$.
\end{theo}
\begin{proof} This follows indeed from
Corollary \ref{corointro} and Theorem \ref{theomain}, since we know
that $H^{20}(Y,\mathbb{Q})_{prim}$ is nonzero by the condition
$h^{11,9}(Y)\not=0$.
\end{proof}
\section{Comments on the ``very ampleness'' assumption \label{seccomment}}
The  very ampleness assumption made previously is too restrictive since there are many more applications
obtained by considering varieties $X$ with the action of a finite group $G$ preserving the line bundle
$L$, and by studying $G$-invariant hypersurfaces $Y\in |L|$, and more precisely
the submotive of $Y$ determined by a projector $\pi\in\mathbb{Q}[G]$. It often happens that
the coniveau of such a submotive is greater than the coniveau of
the whole cohomology of $Y$.

Typically, the quintic Godeaux surfaces $S$ studied in
\cite{voisinpisa} are smooth quintic surfaces, so they have
$h^{2,0}(S)\not=0$. However they are invariant under the Godeaux
action of $G=\mathbb{Z}/5\mathbb{Z}$ and the $G$-invariant part of
$H^{2,0}$ is $0$. If $\pi=\frac{1}{5}\sum_{g\in G}g$ is the
projector onto the $G$-invariant part, we thus have
$H^{2,0}(S)^\pi=0$ so the Hodge coniveau of $H^2(S,\mathbb{Q})^\pi$
is $1$. The Lefschetz theorem on $(1,1)$-classes then says that the
cohomology $H^2(S,\mathbb{Q})^\pi$ consists of classes of
 $1$-cycles and it easily implies that it is parameterized by
 $1$-cycles in the sense of Definition \ref{defiparam}. Similarly, the case of cubic fourfolds
invariant under a finite group acting trivially on $H^{3,1}(X)$ is
studied in \cite{fulie}. In this case, the projector to be
considered is $1-\pi_G$, where $\pi_G$ is again the projector onto
the $G$-invariant part. As $1-\pi_G$ acts as $0$ on $H^{3,1}(X)$,
the Hodge structure on $H^4(X,\mathbb{Q})_{prim}^{1-\pi_G}$ is
trivial of type $(2,2)$. As the Hodge conjecture is satisfied by
cubic fourfolds (see \cite{murre}, \cite{zucker}, or
\cite{voisinweyl} for the integral coefficients version), one gets
that the cohomology
 $H^4(X,\mathbb{Q})_{prim}^{1-\pi_G}$ consists of classes of $2$-cycles, and it implies as above
 that it is parameterized by $2$-cycles in the sense of Definition \ref{defiparam}, while  for the whole cohomology
 $H^4(X,\mathbb{Q})$, it is only parameterized by $1$-cycles.

On the other hand, the linear system of $G$-invariant hypersurfaces is clearly not very
ample, so Theorem \ref{theomain} a priori does not apply.
Let us explain the variants of Theorem \ref{theomain} which will apply to
the situations above.
First of all we have the following:
\begin{prop} \label{propmild} Let $X$ be smooth projective of dimension $n$ with trivial Chow groups,
and $L$ be an ample line bundle on $X$. Assume that

i)  The cohomology
$H^{n-1}(Y_t,\mathbb{Q})_{prim},\,t\in B$, is nonzero and is parameterized by algebraic cycles of dimension $c$.

Then the conclusion of Theorem \ref{theomain} still holds, namely
$CH_i(Y_t)_{hom,\mathbb{Q}}=0$ for $i\leq c-1$ if instead of assuming
$L$ very ample, we only assume

ii) The line bundle $L$ is generated by global sections  and the
locus of points $(x,y)\in X\times X$ such that there exists
$(x,y,z)\in \widetilde{X\times X}_\Delta$, where $z$ is a length $2$
subscheme of $X$ with associated cycle $x+y$
 imposing only one condition to $H^0(X,L)$, has codimension
$>n$ in $X\times X$.
\end{prop}
\begin{rema}{\rm Note that $L$ being ample, the morphism $\phi_L:X\rightarrow \mathbb{P}^N$
given by sections of $L$ is  finite, so a priori   the locus
appearing in ii)
  has codimension
$\geq n$ in $X\times X$. We want that, away from the diagonal, this
locus has codimension $>n$, which is equivalent to saying that
$\phi_L$ is generically $1$-to-$1$ on its image. Our condition along
the diagonal is automatic since it says that $\phi_L$ is generically
an immersion. }
\end{rema}
\begin{proof}[Proof of Proposition \ref{propmild}] Indeed, going through the proof of Theorem \ref{theomain}, we see that we used
the condition that $L$ is very ample to say that
$\mathcal{Y}\times_B\mathcal{Y}$ has a smooth projective completion
\begin{eqnarray}\label{eqM} M=\{((x,y,z),f),\,(x,y,z)\in \widetilde{X\times X}_\Delta,\,f\in |L|,\,f_{\mid z}=0\}
\end{eqnarray}
which is a projective bundle over
$\widetilde{X\times X}_\Delta$. If $L$ is not very ample,
then $M$ defined in (\ref{eqM}) is not anymore a projective bundle over $\widetilde{X\times X}_\Delta$
via
the first projection but we can as in \cite{fulie} overcome this problem
by simply blow-up   $\widetilde{X\times X}_\Delta$ along the sublocus
where the length $2$ subscheme $z$ of $X$ does not impose independent conditions to
$|L|$, until we get
a smooth projective variety
$X'\rightarrow \widetilde{X\times X}_\Delta$ together with a projective bundle
$M'\rightarrow X'$, where $M'$ maps birationally to
$M$ and a Zariski open set $M'_0$ of ${M'}$ admits a dominating proper map
$$\phi_0:M'_0
\rightarrow{\mathcal{Y}\times_B\mathcal{Y}}.$$ Namely letting
$\pi:{M}'\rightarrow |L|$ be the composition of the map
$\tau':{M}'\rightarrow M$ and the second projection $M\rightarrow
|L|$, we can define $M'_0$ as $\pi^{-1}(B)$ and $\phi_0$ is simply
the restriction to $M'_0$ of the composition $\phi$ of
$\tau':{M}'\rightarrow M$ and of the natural map $((x,y,z),f)\mapsto
((x,y),f)$ from $M$ to
$\overline{\mathcal{Y}}\times_{|L|}\overline{\mathcal{Y}}$, where
$\overline{\mathcal{Y}}$ is the universal hypersurface over $|L|$.

Under assumption i), we conclude as in the proof of Theorem
\ref{theomain} that there is a cycle
\begin{eqnarray}\label{Trep}
R:=\overline{\widetilde{\mathcal{Z}}}-\overline{\widetilde{\Delta_{\mathcal{Y}/B}}}-
p^*(\Gamma) \in CH^{n-1}(M')_\mathbb{Q},
\end{eqnarray}
where $\Gamma\in CH^{n-1}(X\times X)_\mathbb{Q}$, and $\mathcal{Z}$
is a codimension
 $n-1$ cycle in $\mathcal{Y}\times_B\mathcal{Y}$ which is supported on
 $\mathcal{W}\times _B\mathcal{W}$, ${\rm codim}\,\mathcal{W}\geq c$,
 such that
the image   $\tau''_{b*}\circ i_b^*(R)\in
CH^{n-1}(\mathcal{Y}_b\times \mathcal{Y}_b)_\mathbb{Q}$ is
cohomologous to $0$, for any $b\in B$. Here the
$\widetilde{\,\cdot\,}$ means that we take the pull-back of the
considered cycles via $\phi^*_0$ and the $\overline{\,\,\cdot\,\,}$
means that we extend the cycles from $M'_0$ to $M'$. The  map $i_b$
is the inclusion of the fiber $\widetilde{\mathcal{Y}_b\times
\mathcal{Y}_b}$ of $\pi$ in $M'$ and the map
$\tau''_b:\widetilde{\mathcal{Y}_b\times \mathcal{Y}_b}\rightarrow
{\mathcal{Y}_b\times \mathcal{Y}_b}$ is the restriction of $\phi_0$
to $\widetilde{\mathcal{Y}_b\times \mathcal{Y}_b}$.

Recall now that $M'$ is a projective bundle over $X'$ which itself
is obtained by blowing up $\widetilde{X\times X}_\Delta$ along
subloci whose images in $X\times X$ are of codimension $>n$, hence
of dimension $<n$ and thus intersect the general
${\mathcal{Y}_b\times \mathcal{Y}_b}$ along a closed algebraic
subset of dimension $< n-1$, since $|L|$ is base-point free. In
particular, we have a morphism $p':M'\rightarrow X\times X$, giving
an inclusion $CH(X\times X)_\mathbb{Q}\rightarrow
CH(M')_\mathbb{Q}$. It is immediate that the morphism
$\tau''_{b*}\circ i_b^*$ is a morphism of $CH(X\times
X)_\mathbb{Q}$-modules. By the general facts concerning the Chow
groups of a projective bundle and a blow-up, we can  write any
element of $CH^{n-1}(M')_\mathbb{Q}$ as a polynomial with
coefficients in  the ring $CH(X\times X)_\mathbb{Q}$ in the
following generators:
\begin{enumerate}
\item \label{1} the class $h=c_1(\mathcal{O}_{M'}(1))$, where the line bundle $\mathcal{O}_{M'}(1)$ is
the pull-back of $\mathcal{O}_{|L|}(1)$ to $M'$ so that $h_{\mid
\widetilde{\mathcal{Y}_b\times \mathcal{Y}_b}}$ is $0$ and thus $
i_b^*(h^k)=0$ for all $k>0$;

\item \label{2} the class  $\delta$, which is the bull-back to $M'$ of the exceptional divisor of
$\widetilde{X\times X}_\Delta$ over the diagonal. The divisor
$\delta$ restricts to the exceptional divisor of
$\widetilde{\mathcal{Y}_b\times \mathcal{Y}_b}$ and the only  power
$\delta^{k},\,0<k\leq n-1$ mapping  to a nonzero element of
$CH^{n-1}(\mathcal{Y}_b\times \mathcal{Y}_b)_\mathbb{Q}$ via
$\tau''_{b*}\circ i_b^*$ is $\delta^{n-1}$, since the other terms
$\delta^k$, with $k<n-1$ will be contracted to the diagonal of
$\mathcal{Y}_b$ via the blow-down map
$\widetilde{\mathcal{Y}_b\times \mathcal{Y}_b}\rightarrow
\mathcal{Y}_b\times \mathcal{Y}_b$.

\item \label{4} Cycles  of codimension $\leq n-1$ supported on the other exceptional divisors of
the blow-up map $X'\rightarrow \widetilde{X\times X}_\Delta$. Any
such cycle will be sent to $0$ in $CH^{n-1}(\mathcal{Y}_b\times
\mathcal{Y}_b)_\mathbb{Q}$ by the map $\tau''_{b*}\circ i_b^*$ since
its intersection with  $\widetilde{\mathcal{Y}_b\times
\mathcal{Y}_b}$  is supported over a sublocus of
$\mathcal{Y}_b\times \mathcal{Y}_b$ of dimension $<n-1$.
\end{enumerate}
Writing the cycle $R$ in (\ref{Trep}) using these generators,
 it follows from this enumeration   that the analogue of Proposition
\ref{lecrux} still holds in our situation, since the extra cycles in
$CH^{n-1}(M')$  appearing in \ref{4} above  vanish in
$CH^{n-1}(\mathcal{Y}_b\times \mathcal{Y}_b)_\mathbb{Q}$, so that we
can simply by modifying $R$ if necessary assume they do not appear.
The classes of the form $\delta^k p^*Z$, for $0<k<n-1$, can be
ignored for the same reason and we conclude that $$\tau''_{b*}\circ
i_b^*(R)=\mu \Delta_{\mathcal{Y}_b}+\Gamma_{\mid \mathcal{Y}_b\times
\mathcal{Y}_b}\,\,{\rm in}\,\,CH^{n-1}(\mathcal{Y}_b\times
\mathcal{Y}_b)_\mathbb{Q}$$ for some cycle $\Gamma\in
CH^{n-1}(X\times X)_\mathbb{Q}$. On the other hand, our assumption
is that $\tau''_{b*}\circ i_b^*(R)$ is cohomologous to $0$. The
assumption made in i) that the cohomology
$H^{n-1}(\mathcal{Y}_t,\mathbb{Q})_{prim},\,t\in B,$ is nonzero
shows that the diagonal of $\mathcal{Y}_b$ is not cohomologous to
the restriction of a cycle in $X\times X$, and it follows that
$\mu=0$.

As $\tau''_{b*}\circ i_b^*(R)= \mathcal{Z}_b-\Delta_{\mathcal{Y}_b}$
modulo a cycle restricted from $X\times X$, we thus conclude as in
the proof of Theorem \ref{theomain} that there is a codimension
$n-1$ cycle $\gamma$ in $X\times X$ such that $\gamma_{\mid
\mathcal{Y}_b\times
\mathcal{Y}_b}=\Delta_{\mathcal{Y}_b}-\mathcal{Z}_b$ in
$CH^{n-1}(\mathcal{Y}_b\times \mathcal{Y}_b)_\mathbb{Q}$ and the end
of the proof of Proposition \ref{propmild} then works exactly as in
the proof of Theorem \ref{theomain}.

\end{proof}
Proposition \ref{propmild} does not apply to the above mentioned
situation where we replace $|L|$ by some $G$-invariant linear
subsystem $|L|^G$ (or $(\chi,G)$-invariant for some character
$\chi$), where $G$ is a finite group acting on $X$, since then the
(proper transforms of the) graphs of elements of $g\in G$ in
$\widetilde{X\times X}_\Delta$ provide codimension $n$ subvarieties
of $\widetilde{X\times X}_\Delta$ along which the subscheme $z$
imposes at most one condition to $|L|^G$. The best we can assume in
this situation is the following:

\vspace{0.5cm}

 (*)  The linear system $|L|^G:=\mathbb{P}(H^0(X,L)^G)$ has no base-points and the codimension $\leq n$ components
 of the locus of
points in $ \widetilde{X\times X}_\Delta$ parameterizing triples
$(x,y,z)$ such that the length $2$ subscheme $z$ with support $x+y$
imposes only one condition to $H^0(X,L)^G$ is the union of the
(proper transforms of the) graphs of elements of $e\not=g\in G$ (and
 this equality is a scheme theoretic equality generically  along
each of these graphs).

\vspace{0.5cm}

Then we have the following variant of Theorem \ref{theomain}. Let
$X$ be smooth projective with trivial Chow groups, endowed with an
ample line bundle $L$ and an action of the finite group $G$ such
that $L$ is $G$-linearized and  satisfies (*). Let
$\pi=\sum_ga_gg\in\mathbb{Q}[G]$ be a projector  of $G$. For a
general hypersurface $Y\in |L|^G$, $Y$ is smooth and we assume that
$\pi^*$ acts on $H^{n-1}(Y,\mathbb{Q})_{prim}$ as the orthogonal
projector $H^{n-1}(Y,\mathbb{Q})\rightarrow
H^{n-1}(Y,\mathbb{Q})^\pi$.
\begin{theo}\label{theovariant}   Assume the following:

(i)  For the general hypersurface $Y\in |L|^G$  the cohomology
$H^{n-1}(Y,\mathbb{Q})_{prim}^\pi$ is parameterized by cycles of
dimension $c$.

(ii)  The primitive components $g^*\in
End_\mathbb{Q}\,(H^{n-1}(Y,\mathbb{Q})_{prim})$ of the cohomology
classes of the graphs of elements of $g$ are linearly independent
over $\mathbb{Q}$.

Then the groups $CH_i(Y)^\pi_{\mathbb{Q},hom}$ are trivial for
$i\leq c-1$.

\end{theo}
\begin{proof} The proof is a  generalization of the proofs of
Theorem \ref{theomain} and Proposition \ref{propmild}. Let $B\subset
|L|^G$ be the open set parameterizing smooth invariant hypersurfaces
and let  $\Delta_{\pi,b}=\sum_ga_g\Gamma_g\subset Y_b\times Y_b$,
where $\Gamma_g$ is the graph of $g$ acting on $Y_b$; let
$\Delta_{\pi,b,prim}$ be the primitive part of $\Delta_{\pi,b}$,
obtained by correcting $\Delta_{\pi,b}$ by the restriction to
$Y_b\times Y_b$ of a $\mathbb{Q}$-cycle of $X\times X$, in such a
way that $[\Delta_{\pi,b,prim}]^*$ acts as the orthogonal projector
onto $H^{n-1}(Y_b,\mathbb{Q})^\pi_{prim}$.

Our assumption that $H^{n-1}(Y_b,\mathbb{Q})^\pi_{prim}$ is parameterized by
algebraic cycles of codimension $c$ implies that there exists a codimension
$c$ closed algebraic subset $W_b\subset Y_b$ and a $n-1$-cycle
$Z_b\subset W_b\times W_b$ such that
$Z_b$ is cohomologous to $\Delta_{\pi,b,prim}$ in
$Y_b\times Y_b$.

We then spread these data over $B$ and get a codimension $c$
subvariety $\mathcal{W}\subset \mathcal{Y}$, where
$f:\mathcal{Y}\rightarrow B$ is the universal family, and a cycle
$\mathcal{Z}$ supported on $\mathcal{W}\times_B\mathcal{W}$ such
that
$$\mathcal{Z}-\Delta_{\pi,\mathcal{Y}/B,prim}$$
has its restriction cohomologous to $0$ on the fibers
$\mathcal{Y}_b\times \mathcal{Y}_b$ of the map
$(f,f):\mathcal{Y}\times_B\mathcal{Y}\rightarrow B$.

We now have to prove the analogue of Proposition \ref{lecrux}. As in
the proof of Proposition \ref{propmild}, the difficulty comes from
the fact that the variety
$$M:=\{((x,y,z),\sigma)\in \widetilde{X\times X}_\Delta\times|L|^G,\,\sigma_{\mid z}=0\}$$
is no longer a projective bundle over $\widetilde{X\times X}_\Delta$
due to the lack of very ampleness of the $G$-invariant linear system
$|L|^G$. In the case of Proposition \ref{propmild}, we had a smooth
projective model $X'$ of $\widetilde{X\times X}_\Delta$ obtained by
blowing-up $\widetilde{X\times X}_\Delta$ along subloci of
codimension $>n$, on which we analyzed the conveniently defined
extension $R$ of the cycle
$\mathcal{Z}-\Delta_{\pi,\mathcal{Y}/B,prim}$ (first by pull-back
under blow-up to
$\widetilde{\mathcal{Y}\times_B\mathcal{Y}}_\Delta$, and then by
extension to the projective completion $M'$). In our new situation,
the only new feature lies in the fact that in order to get the
projective bundle $M'\rightarrow X'$,
 we have to blow-up in $\widetilde{X\times X}_\Delta$   the graphs of $g\in G$
 which are of codimension $n$ and intersect $\mathcal{Y}_b\times \mathcal{Y}_b$ along
 a codimension $n-1$ locus, namely the graph $\Gamma_g$ of $g$ acting on
 $\mathcal{Y}_b$. As in the proof of Proposition \ref{propmild},  further blow-ups may be needed in order to construct
 the model $T'$, but they are over closed algebraic subsets
 of $X\times X$ of codimension $>n$.

For any codimension $n-1$ cycle supported in $M'$ supported in an
exceptional divisor of the map $X'\rightarrow X\times X$ over ${\rm
graph}\,(g)$, its image in $CH^{n-1}(\mathcal{Y}_b\times
\mathcal{Y}_b)_\mathbb{Q}$ is a multiple of $\Gamma_g$.

With the same notations as in the proof of Proposition
\ref{propmild}, we write our cycle $T\in CH^{n-1}(M')_\mathbb{Q}$ as a sum
   $$T=P(h,\delta_g)+A,$$
   where $P$ is a polynomial in the variables
$h,\,\delta_g,\,g\in G,$ whose coefficients are pull-backs of cycles
on $X\times X$, and $A$ is a cycle supported on an exceptional
divisor of $X'\rightarrow X$ over a closed algebraic subset of
$X\times X$ of codimension $>n$. Here $h=c_1(\mathcal{O}_{M'}(1))$,
where the line bundle $\mathcal{O}_{M'}(1)$ comes from
$\mathcal{O}_{|L|^G}(1)$ and thus restricts to $0$ on the fibers of
$M'\rightarrow  |L|^G$, which are birationally equivalent to
$\mathcal{Y}_b\times \mathcal{Y}_b$. The divisors $\delta_b$ are the
exceptional divisors over the generic points of the graphs ${\rm
graph}\,g$.

We now recall that the cycle $R$ maps, via the natural
correspondence $\tau_{b*}\circ i_b^*$ between $M'$ and
$\mathcal{Y}_b\times \mathcal{Y}_b$, to
$\mathcal{Z}_b-\Delta_{\pi,b,prim}\in CH^{n-1}(\mathcal{Y}_b\times
\mathcal{Y}_b)_\mathbb{Q}$, where $\mathcal{Z}_b$ is supported on
$\mathcal{W}_b\times\mathcal{W}_b$, with ${\rm codim}\,
\mathcal{W}_b\subset \mathcal{Y}_b\geq c$.

In our polynomial $P(h,\delta_g)$, only the terms of degree $0$ in
$h$ can be mapped by $\tau_{b*}\circ i_b^*$ to a nonzero element in
$CH^{n-1}(\mathcal{Y}_b\times \mathcal{Y}_b)$ and concerning the
powers of $\Delta_g$, only the terms of degree $n-1$ in $\delta_g$
can be mapped to a nonzero element in $CH^{n-1}(\mathcal{Y}_b\times
\mathcal{Y}_b)_\mathbb{Q}$ (and they are then mapped to the class of
$\Gamma_g$ in $CH^{n-1}(\mathcal{Y}_b\times
\mathcal{Y}_b)_\mathbb{Q}$). The monomials of degree $\leq n-1$
involving at least two of the $\delta_g$ will also be annihilated by
$\tau_{b*}\circ i_b^*$ since their images will be supported on
$\Gamma_g\cap \Gamma_{g'}$ which has dimension $<n-1$. Hence we can
assume that
$$R=R_0+\sum_g\lambda_g\delta_g^{n-1},$$
 where $R_0$ is the pull-back to $M'$ of a cycle on $X\times
X$, without changing  the image $\tau_{b*}\circ i_b^*(R)\in
CH^{n-1}(\mathcal{Y}_b\times \mathcal{Y}_b)_\mathbb{Q}$. We thus
have
\begin{eqnarray}
\label{eqfinfinfin}\tau_{b*}\circ i_b^*(R)=R_{0\mid
\mathcal{Y}_b\times
\mathcal{Y}_b}+(-1)^n\sum_g\lambda_g\Gamma_g^{n-1}\,\,{\rm
in}\,\,CH^{n-1}(\mathcal{Y}_b\times \mathcal{Y}_b)_\mathbb{Q}
\end{eqnarray}
We know that $\tau_{b*}\circ i_b^*(R)$ is cohomologous to $0$ in
$\mathcal{Y}_b\times \mathcal{Y}_b$. As we made the assumption that
the endomorphisms
$\Gamma_{g,b*}:H^{n-1}(\mathcal{Y}_b,\mathbb{Q})_{prim}\rightarrow
H^{n-1}(\mathcal{Y}_b,\mathbb{Q})_{prim}$ are linearly independent,
we conclude from (\ref{eqfinfinfin}) that  all $\lambda_g$ vanish,
so that $R=R_0$. As we have
$$\tau_{b*}\circ i_b^*(R)=\mathcal{Z}_b-\Delta_{\pi,b,prim}\,\,{\rm
in}\,\, CH^{n-1}(\mathcal{Y}_b\times Y_b)_\mathbb{Q},$$ we conclude
that
$$\mathcal{Z}_b-\Delta_{\pi,b,prim}-R_{0\mid Y_b\times Y_b}=0\,\,{\rm\in }\,\,CH^{n-1}(\mathcal{Y}_b\times \mathcal{Y}_b)_\mathbb{Q},$$
where we recall that $\mathcal{Z}_b$ is supported on
$\mathcal{W}_b\times \mathcal{W}_b$ with ${\rm
codim}\,\mathcal{W}_b\geq c$. The argument explained in the
introduction then allows to conclude that
$CH_i(\mathcal{Y}_b)^\pi_{hom,\mathbb{Q}}=0$ for $i<c$.

\end{proof}
We refer to \cite{voisingenhodgebloch} for further potential
applications of the general strategy developed in Theorems
\ref{theomain}, \ref{theovariant}. Let us just mention one
challenging example. In \cite{voisinpisa}, the case of quintic
hypersurfaces in $\mathbb{P}^4$ invariant under the involution
acting by $(-1,-1,+1,+1,+1)$ on homogeneous coordinates is studied.
The involution acts as the identity
 on $H^{3,0}(X)$ and it is proved that the antiinvariant part of $H^3(X,\mathbb{Q})$
is parameterized by $1$-cycles. Theorem \ref{theovariant} above
(applied to the blow-up of $\mathbb{P}^4$ along the line
$\{X_2=X_3=X_4=0\}$ to avoid base-points) then implies that
$CH_0(X)^-$ is equal to $0$, a result which was already obtained in
\cite{voisinpisa}. The next case to study would be that of a sextic
hypersurface in $\mathbb{P}^5$ defined by an equation invariant
under the involution $i$ acting on homogeneous coordinates by
\begin{eqnarray}\label{typeinv}i^*(X_0,\ldots,X_5)=(-X_0,-X_1,-X_2,X_3,X_4,X_5).
\end{eqnarray}
This involution acts by $-Id$ on $H^{4,0}(X)$ and thus the
cohomology $H^4(X,\mathbb{Q})^+$ invariant under the involution has
Hodge coniveau $1$, so is expected to be parameterized by
$1$-cycles. Assuming this is true, then Theorem \ref{theovariant}
would imply that the invariant part $CH_0(X)^+_0$ of the group of
$0$-cycles of degree $0$ on $X$ is $0$. Indeed, Remark
\ref{rema30jan} applies to the very general invariant hypersurface
in this case, by standard infinitesimal variations of Hodge
structure arguments. This shows that if for the general invariant
hypersurface $X$ as above, $H^4(X,\mathbb{Q})^+$ is of geometric
coniveau $1$, then it is parameterized by $1$-cycles in the sense of
Definition \ref{defiparam}.

This example is particularly interesting because it relates to the
following question asked and studied in \cite[Section 3]{voisinsym}:
For any variety $Y$, we have the map
$$\mu_Y: CH_0(Y)_{hom}\otimes CH_0(Y)_{hom}\rightarrow CH_0(Y\times Y)$$
$$z\otimes z'\mapsto p_1^*z\cdot p_2^* z',$$
and the map
$$\mu^-_Y:CH_0(Y)_{hom}\otimes CH_0(Y)_{hom}\rightarrow CH_0(Y\times Y),$$
$$z\otimes z'\mapsto p_1^*z\cdot p_2^* z'-p_1^*z'\cdot p_2^*z.$$
Let now $S$ be a smooth
projective $K3$ surface.
\begin{question}\label{question} Is it true that
the map
$\mu_S^-$
is $0$?
\end{question}
This is implied by the generalized Bloch conjecture since the space
$H^{4,0}(S\times S)^-$ of  holomorphic $4$-forms  on $S\times S$
antiinvariant under the involution exchanging the factors is $0$.
(Note that there are nonzero antiinvariant holomorphic $2$-forms on
$S\times S$, but they are of the form $p_1^*\omega-p_2^*\omega$,
where $\omega\in H^{2,0}(S)$, while the $0$-cycles  in the image of
$\mu_S^-$ are annihilated by $p_{1*}$ and $p_{2*}$.)

The precise relation between Question \ref{question} and
$i$-invariant
 $CH_0$ groups of sextic hypersurfaces invariant under an involution
 $i$  of the type (\ref{typeinv}) is the following : the Shioda construction (see \cite{shioda})
shows that if $C\subset \mathbb{P}^2$ is a plane curve of degree $6$
defined by a polynomial equation $f(X_0,X_1,X_2)$, the sextic
fourfold $X$ defined by the equation $f(X_0,X_1,X_2)-f(Y_0,Y_1,Y_2)$
is rationally dominated by the product $\Sigma\times \Sigma$, where
$\Sigma$ is the sextic surface in $\mathbb{P}^3$ with equation
$U^6=f(X_0,X_1,X_2)$. The rational map
$\Phi:\Sigma\times\Sigma\dashrightarrow X$ is explicitly given by
$$\Phi((x,u),(y,v))=(vx,uy).$$
It makes $X$ birationally equivalent to the quotient of
$\Sigma\times \Sigma$ by $G=\mathbb{Z}/6\mathbb{Z}$, where we choose
an isomorphism $g\mapsto \zeta$ between  $G$ and the group of $6$th
roots of unity and the actions of $g\in G$ on $\Sigma$ and
$\Sigma\times \Sigma$ are given by
$$g(x,u)=(x,\zeta u),$$
$$g((x,u),(y,v))=(g(x, u),g(y, v)).$$
Note now that the $K3$ surface $S$, which is defined as the double
cover of $\mathbb{P}^2$ ramified along $C$, is also the quotient of
$\Sigma$ by the action of $ \mathbb{Z}/3\mathbb{Z}\subset
\mathbb{Z}/6\mathbb{Z}$. Let $p:\Sigma\rightarrow S$ be the quotient
map. We now have
\begin{lemm} \label{profin} Via the map
$\Phi_*\circ (p,p)^*$, the group ${\rm Im}\,\mu_S:(CH_0(S)_0\otimes
CH_0(S)_0\rightarrow CH_0(S\times S)$ embeds into $CH_0(X)_0$, and
the image ${\rm Im}\,\mu_S^-$ embeds into the invariant part of
$CH_0(X)_0$ under the involution $i$ (which is of the type
(\ref{typeinv})) acting on coordinates by
$$i(X_0,X_1,X_2,Y_0,Y_1,Y_2)=\sqrt{-1}(Y_0,Y_1,Y_2,-X_0,-X_1,-X_2),$$
which leaves the equation of $X$ invariant.

\end{lemm}
\begin{proof} The Shioda rational map $\Phi$ is the quotient map by the
group $G=\mathbb{Z}/6\mathbb{Z}$. So
for a $0$-cycle $z\in CH_0(\Sigma\times \Sigma)$, we have
$$\Phi^*(\Phi_*(z))=\sum_{g\in G}g^*z.$$
Let now $z=\sum_ipr_1^*z_i\cdot pr_2^*z'_i,\,{\rm deg}\,z_i={\rm deg}\,z'_i=0,$ be an element of
 ${\rm Im}\,\mu$. Then
denoting by $j$ the involution of $S$ over $\mathbb{P}^2$, we have
$j^*z_i=-z_i$, $j^*z'_i=-z'_i$, so that $(j,j)^*(z)=z$. It
immediately follows that $(p,p)^*z$ is invariant under $G$, so that
$\sum_{g\in G}g^*((p,p)^*z)=6(p,p)^*z$, which proves the injectivity
since $(p,p)^*$ is injective.

Let us now check that the cycles in $\Phi_*({\rm Im}\,\mu^-)$ are invariant under
$i$. Indeed, elements of $(p,p)^*({\rm Im}\,\mu^-)$ are antiinvariant under the involution
$\tau$ acting on $\Sigma\times \Sigma$ exchanging factors.
On the other hand, elements of ${\rm Im}\,\mu^-$
 are also antiinvariant under the involution $(Id,j)$ acting
 on $S\times S$. It follows that for
 $z\in {\rm Im}\,\mu^-$, one has $\tau^*((p,p)^*((Id,j)^*(z)))=z$.
 Applying $\Phi_*$, we get that $\Phi_*((p,p)^*z)$ is invariant under $i$.
\end{proof}

In conclusion, if we were able to prove that for the sextic
fourfolds $X$ invariant under the involution $i$ of the type
(\ref{typeinv}), the $i$-invariant part of $H^4(X,\mathbb{Q})$ is
parameterized by $1$-cycles, then by Theorem \ref{theovariant}, we
would get that $CH_0(X)^+_0=0$ and by Lemma \ref{profin}, we would
conclude that the map $\mu_S^-$ is $0$, thus solving  Question
\ref{question} for $K3$ surfaces which are ramified double covers
$\mathbb{P}^2$.

\vspace{0.5cm}

 {\bf Thanks.} {\it  I thank Lie Fu for his careful
reading}.

 Centre de math\'{e}matiques Laurent Schwartz

91128 Palaiseau C\'{e}dex

 France

\smallskip
 voisin@math.polytechnique.fr
    \end{document}